%% file: paper.tex
\documentclass[a4paper,USenglish,cleveref, autoref, thm-restate]{lipics-v2021}

\hideLIPIcs  

\title{The Functor of Points Approach to Schemes in Cubical Agda} 

\author{Max Zeuner}{Department of Mathematics, Stockholm University, Sweden}{zeuner@math.su.se}{https://orcid.org/0000-0003-3092-8144}{}

\author{Matthias Hutzler}{Department of Computer Science and Engineering, Gothenburg University, Sweden}{matthias-hutzler@posteo.net}{}{}

\authorrunning{M. Zeuner and M. Hutzler} 

\Copyright{Max Zeuner and Matthias Hutzler} 



\category{} 

\supplement{}
\supplementdetails[subcategory={Agda Source Code}]{Software}{https://github.com/agda/cubical/blob/master/Cubical/Papers/FunctorialQcQsSchemes.agda} 

\funding{This paper is based upon research supported by the Swedish
  Research Council (SRC, Vetenskapsrådet) under Grant
  No.~2019-04545.}

\acknowledgements{
  We would like to thank Felix Cherubini, Thierry Coquand and Anders Mörtberg
  for their support and feedback throughout this project.
}

\nolinenumbers 

\EventEditors{John Q. Open and Joan R. Access}
\EventNoEds{2}
\EventLongTitle{42nd Conference on Very Important Topics (CVIT 2016)}
\EventShortTitle{CVIT 2016}
\EventAcronym{CVIT}
\EventYear{2016}
\EventDate{December 24--27, 2016}
\EventLocation{Little Whinging, United Kingdom}
\EventLogo{}
\SeriesVolume{42}
\ArticleNo{23}

\include{macros}
\begin{document}

\maketitle

\begin{abstract}
  We present a formalization of quasi-compact and quasi-separated
  schemes (qcqs-schemes) in the \CubicalAgda~proof assistant. We follow
  Grothendieck's functor of points approach, which defines schemes,
  the quintessential notion of modern algebraic geometry,
  as certain well-behaved functors from commutative rings to sets.
  This approach is often regarded as conceptually simpler than the standard approach of
  defining schemes as locally ringed spaces,
  but to our knowledge it has not yet been adopted in
  formalizations of algebraic geometry.
  We build upon a previous formalization of the so-called
  Zariski lattice associated to a commutative ring in order to define the notion
  of compact open subfunctor. This allows for a concise
  definition of qcqs-schemes, streamlining the usual
  presentation as e.g.\ given in the standard textbook of Demazure and Gabriel.
  It also lets us obtain a fully constructive
  proof that compact open subfunctors of affine schemes are qcqs-schemes.
\end{abstract}

\input{introduction}
\input{background}
\input{zfunctors}
\input{coverage}
\input{compactopens}
\input{opensubschemes}

\input{conclusion}

\bibliography{refs}
\end{document}

%% file: macros.tex
\usepackage{stmaryrd}
\usepackage{relsize}
\usepackage{nicefrac}
\usepackage{tikz-cd}
\usepackage{graphicx}
\usepackage{mathtools}
\usepackage{quiver}
\usepackage{xspace}

\newcommand{\teletype}[1]{\ensuremath{\mathtt{#1}}}
\newcommand{\systemname}[1]{\teletype{\color{darkgray}#1}\xspace}

\newcommand{\Agda}{\systemname{Agda}}

\newcommand{\agdaCubical}{\systemname{agda/cubical}}
\newcommand{\Coq}{\systemname{Coq}}

\newcommand{\CubicalAgda}{\systemname{Cubical} \systemname{Agda}}

\newcommand{\Lean}{\systemname{Lean}}
\newcommand{\mathlib}{\systemname{mathlib}}
\newcommand{\IsabelleHOL}{\systemname{Isabelle/HOL}}

\newcommand{\UniMath}{\systemname{UniMath}}
\definecolor{Revolutionary}{RGB}{232,70,68}

\definecolor{dkblue}{rgb}{0,0.1,0.5}
\definecolor{lightblue}{rgb}{0,0.5,0.5}
\definecolor{dkgreen}{rgb}{0,0.6,0}
\definecolor{dkbrown}{rgb}{0.4,0,0}
\definecolor{dkviolet}{rgb}{0.6,0,0.8}

\usepackage{agda}

\newcommand{\func}[1]{{\AgdaFunction{#1}}}

\usepackage{catchfilebetweentags}

\usepackage{newunicodechar}
\newunicodechar{≃}{\ensuremath{\simeq}}
\newunicodechar{≅}{\ensuremath{\cong}}
\newunicodechar{∈}{\ensuremath{\in}}
\newunicodechar{⁻}{\ensuremath{^{-}}}
\newunicodechar{₋}{\ensuremath{_{-}}}
\newunicodechar{ᴹ}{\ensuremath{_{M}}}
\newunicodechar{ᴺ}{\ensuremath{_{N}}}
\newunicodechar{≡}{\ensuremath{\equiv}}
\newunicodechar{≈}{\ensuremath{\approx}}
\newunicodechar{∼}{\ensuremath{\sim}}
\newunicodechar{λ}{\ensuremath{\lambda}}
\newunicodechar{⊎}{\ensuremath{\uplus}}
\newunicodechar{∷}{\ensuremath{::}}
\newunicodechar{ℓ}{\ensuremath{\ell}}
\newunicodechar{⃗}{\ensuremath{^{\rightarrow}}}
\newunicodechar{⃖}{\ensuremath{^{\leftarrow}}}
\newunicodechar{ᵢ}{\ensuremath{_i}}
\newunicodechar{⟨}{\ensuremath{\langle}}
\newunicodechar{⟩}{\ensuremath{\rangle}}
\newunicodechar{α}{\ensuremath{\alpha}}
\newunicodechar{β}{\ensuremath{\beta}}
\newunicodechar{θ}{\ensuremath{\theta}}
\newunicodechar{φ}{\ensuremath{\varphi}}
\newunicodechar{ψ}{\ensuremath{\psi}}
\newunicodechar{χ}{\ensuremath{\chi}}
\newunicodechar{η}{\ensuremath{\eta}}
\newunicodechar{ε}{\ensuremath{\varepsilon}}
\newunicodechar{ι}{\ensuremath{\iota}}
\newunicodechar{Σ}{\ensuremath{\Sigma}}
\newunicodechar{∀}{\ensuremath{\forall}}
\newunicodechar{∃}{\ensuremath{\exists}}
\newunicodechar{ℕ}{\ensuremath{\mathbb{N}}}
\newunicodechar{→}{\ensuremath{\to}}
\newunicodechar{⇒}{\ensuremath{\Rightarrow}}
\newunicodechar{⊎}{\ensuremath{\uplus}}
\newunicodechar{⋆}{\ensuremath{*}}
\newunicodechar{·}{\ensuremath{\cdot}}
\newunicodechar{∅}{\ensuremath{\emptyset}}
\newunicodechar{×}{\ensuremath{\times}}
\newunicodechar{¬}{\ensuremath{\lnot}}
\newunicodechar{Θ}{\ensuremath{\theta}}
\newunicodechar{₁}{\ensuremath{_1}}
\newunicodechar{₂}{\ensuremath{_2}}
\newunicodechar{¹}{\ensuremath{^1}}
\newunicodechar{ₗ}{\ensuremath{_l}}
\newunicodechar{ₚ}{\ensuremath{_p}}
\newunicodechar{ᵤ}{\ensuremath{_u}}
\newunicodechar{ˣ}{\ensuremath{^{\times}}}
\newunicodechar{∪}{\ensuremath{\cup}}
\newunicodechar{⋂}{\ensuremath{\cap}}
\newunicodechar{ᴸ}{\ensuremath{{}^{\color{black}\leftarrow}}}  
\newunicodechar{ᴿ}{\ensuremath{{}^{\color{black}\rightarrow}}} 
\newunicodechar{𝓞}{\ensuremath{\mathcal{O}}}
\newunicodechar{ᴮ}{\ensuremath{^B}}
\newunicodechar{ᴰ}{\ensuremath{^D}}
\newunicodechar{≼}{\ensuremath{\preceq}}
\newunicodechar{⊆}{\ensuremath{\subseteq}}
\newunicodechar{Π}{\ensuremath{\Pi}}
\newunicodechar{⦃}{\ensuremath{\lBrace}}
\newunicodechar{⦄}{\ensuremath{\rBrace}}
\newunicodechar{∣}{\ensuremath{\mid}}
\newunicodechar{ⁿ}{\ensuremath{^n}}
\newunicodechar{≥}{\ensuremath{\ge}}
\newunicodechar{ℙ}{\ensuremath{\mathbb{P}}}
\newunicodechar{𝓞}{\ensuremath{\mathcal{O}}}
\newunicodechar{ᴰ}{\ensuremath{^D}}
\newunicodechar{≼}{\ensuremath{\preceq}}
\newunicodechar{𝔞}{\ensuremath{\mathfrak{a}}}
\newunicodechar{𝔟}{\ensuremath{\mathfrak{b}}}
\newunicodechar{𝓕}{\ensuremath{\mathcal{F}}}
\newunicodechar{𝓟}{\ensuremath{\mathcal{P}}}
\newunicodechar{∙}{\ensuremath{\bullet}}
\newunicodechar{∎}{\ensuremath{\qed}}
\newunicodechar{√}{\ensuremath{\surd}}

\newunicodechar{∧}{\ensuremath{\wedge}}
\newunicodechar{∨}{\ensuremath{\vee}}
\newunicodechar{∼}{\ensuremath{\sim}}
\newunicodechar{≢}{\ensuremath{\nequiv}}
\newunicodechar{Π}{\ensuremath{\Pi}}
\newunicodechar{ℤ}{\ensuremath{\mathbb{Z}}}
\newunicodechar{∥}{\ensuremath{\parallel}}
\newunicodechar{ℚ}{\ensuremath{\mathbb{Q}}}
\newunicodechar{₊}{\ensuremath{_+}}
\newunicodechar{∗}{\ensuremath{\ast}}
\newunicodechar{₀}{\ensuremath{_0}}
\newunicodechar{₁}{\ensuremath{_1}}
\newunicodechar{¹}{\ensuremath{^1}}
\newunicodechar{₂}{\ensuremath{_2}}
\newunicodechar{⊥}{\ensuremath{\bot}}
\newunicodechar{⊤}{\ensuremath{\top}}
\newunicodechar{∘}{\ensuremath{\circ}}
\newunicodechar{ρ}{\ensuremath{\rho}}
\newunicodechar{𝔸}{\ensuremath{\mathbb{A}}}
\newunicodechar{⋁}{\ensuremath{\bigvee}}

\newcommand{\Type}{\func{Type}}
\newcommand{\hProp}{\func{hProp}}
\newcommand{\ptrunc}[1]{\func{∥}\,{#1}\,\func{∥}}

\newcommand{\tySigma}[3]{\func{Σ[}~{#1}~\func{∈}~{#2}~\func{]}~{(#3)}}
\newcommand{\tySigmaNoParen}[3]{\func{Σ[}~{#1}~\func{∈}~{#2}~\func{]}~{#3}}
\newcommand{\tyExists}[3]{\func{∃[}~{#1}~\func{∈}~{#2}~\func{]}~{(#3)}}
\newcommand{\tyExistsNoParen}[3]{\func{∃[}~{#1}~\func{∈}~{#2}~\func{]}~{#3}}

\newcommand{\tyPath}[2]{{#1}~\func{≡}~{#2}}

\newcommand{\NatTrans}[2]{{#1}~\func{⇒}~{#2}}
\newcommand{\coBrackets}[1]{\llbracket\,{#1}\,\rrbracket^{\mathsf{co}}}

\newcommand{\rest}[2]{#1\!\!\restriction_{#2}}
\newcommand{\locEl}[2]{#1[\nicefrac{1}{#2}]}
\newcommand{\ZarLat}[1]{\mathcal{L}_{#1}}

\newcommand{\Z}{\mathbb{Z}}
\newcommand{\ZL}{\ensuremath{\mathcal{L}_R}}

\newcommand{\Aone}{\func{𝔸¹}}
\newcommand{\Ofun}{\func{𝓞}}
\newcommand{\ZFunctor}{\func{ℤFunctor}}

\DeclarePairedDelimiterX{\norm}[1]{\lVert}{\rVert}{#1}

%% file: introduction.tex
\section{Introduction}\label{sec: intro}
Algebraic geometry developed as the study of solutions
to systems of polynomial equations.
Objects of interest would e.g.\ be ``affine complex varieties'',
subsets of $\mathbb{C}^n$ that can be described as the common roots
of a finite system of polynomials $p_1,...,p_m\in\mathbb{C}[x_1,...,x_n]$.
The discipline underwent a fundamental transformation during the latter
half of the 20th century with the introduction of \emph{schemes}.
This work was spear-headed by Alexendre Grothendieck and led to many incredible
achievements in geometry and number theory. Schemes can be seen as a generalization of
varieties in several ways,
but their standard presentation as ``locally ringed spaces with an affine cover''
somewhat blurs the connection to classical algebraic geometry,
which can make it hard for students learning algebraic geometry to see
in what sense schemes are ``geometric'' objects at all.

There is, however, a different angle for generalization, where the original
motivation of studying solutions to polynomials keeps a more prominent place.
Take a polynomial with integer coefficients like $x^n+y^n-z^n\in\Z[x,y,z]$.
Fermats last theorem tells us that this polynomial only has the trivial solution
$x=y=z=0$ for $n> 2$. This does of course only hold for solutions in the integers.
We might interpret the same polynomial as living in a polynomial ring
$A[x,y,z]$, where $A$ is now any commutative ring (e.g.\ $\mathbb{C}$),
and ask about solutions in $A$.
The corresponding set of solutions is given by
\begin{align*}
  V_{x^n+y^n-z^n}(A)=\{\,(a_1,a_2,a_3)\in A^3\,\vert\, a_1^n+a_2^n=a_3^n\,\}
\end{align*}
Moreover, given a morphism of rings $\varphi\in\mathsf{Hom}(A,B)$,
we can map a solution in $A$, $(a_1,a_2,a_3)\in V_{x^n+y^n-z^n}(A)$,
to a solution $(\varphi(a_1),\varphi(a_2),\varphi(a_3))\in V_{x^n+y^n-z^n}(B)$ in $B$.
In categorical terms, our polynomial defines a \emph{functor} from the
category of commutative rings to the category of sets, mapping a ring $A$ to
the set $V_{x^n+y^n-z^n}(A)$ of solutions in $A$.

This functor $V_{x^n+y^n-z^n}(\_):\mathsf{CommRing}\to\mathsf{Set}$ turns out to
be a very familiar categorical object. For a ring $A$, homomorphisms
$\mathsf{Hom}(\Z[x,y,z],A)$ are in bijection with $A^3$
(every morphism is determined by its values on $x$, $y$ and $z$)
and this induces a bijection of morphisms
$\mathsf{Hom}(\nicefrac{\Z[x,y,z]}{\langle x^n+y^n-z^n\rangle},A)$
with $V_{x^n+y^n-z^n}(A)$. Now, a functor from $\mathsf{CommRing}$ to $\mathsf{Set}$
is nothing but a presheaf on the opposite category $\mathsf{CommRing}^{op}$.
In this presheaf category we can look at the Yoneda embedding or representable of
the quotient ring
$R=\nicefrac{\Z[x,y,z]}{\langle x^n+y^n-z^n\rangle}$,
which we will denote by $\mathsf{Sp}(R)$.
By the above argument, we get a natural isomorphism of presheaves
$\mathsf{Sp}(R)=\mathsf{Hom}(R,\_)\cong V_{x^n+y^n-z^n}(\_)$.

In the category of functors from commutative rings to sets we can thus study
solutions of systems of integer polynomials by looking at representable functors
of quotients $\nicefrac{\Z[x_1,...,x_n]}{\langle p_1,...,p_m\rangle}$.
Algebraic geometers call these representables
\emph{absolute affine algebraic spaces} \cite{GrothendieckBuffalo}.
From these we can generalize to schemes
(schemes over $\Z$ or absolute schemes to be more precise).
Affine schemes are readily defined as representables of arbitrary commutative rings.
From these we can build general schemes as
presheaves on $\mathsf{CommRing}^{op}$ that are ``local'' and have
an ``open cover'' by affine schemes in some appropriate sense.

Among the proponents of using the functor of points approach as the primary
definition of schemes was Grothendieck himself \cite{GrothendieckBuffalo},
because, in the terms of Lawvere \cite{LawvereMail}, it does not require
``the baggage of prime ideals and the spectral space, sheaves of local rings, coverings and patchings, etc.''
Yet, most standard sources \cite{EisenbudHarris,GoertzWedhorn,Hartshorne,RisingSea}
for students learning algebraic geometry start with precisely this ``baggage''.
To our knowledge, the same can be said for existing formalizations of schemes
\cite{SchemesCoq, SchemesHOL, SchemesLean, SchemesUniMath, ZeunerMortberg23}.
We want to close this gap and present a first formalization
of the functor of points approach.

Admittedly, part of the appeal of schemes as locally ringed spaces
as a formalization target for proof assistants
is that they are such a layered, involved notion,
while at the same time being a point of departure
for formalizing a plethora of interesting research level mathematics.
The first full formalization of schemes in \Lean's \mathlib~by
Buzzard et.\ al.\ \cite{SchemesLean} revealed
certain bottlenecks that occur when defining schemes this way.
As these bottlenecks might be addressed very differently in different proof assistants,
schemes have become somewhat of a benchmark problem, inspiring (partial)
formalizations in \IsabelleHOL~\cite{SchemesHOL},
\Coq's \UniMath~library \cite{SchemesUniMath} and \CubicalAgda~\cite{ZeunerMortberg23}.

It is worth noting that, except for the \CubicalAgda-formalization \cite{ZeunerMortberg23},
all of the above formalizations are non-constructive as they follow
the presentation of Hartshorne's standard ``Algebraic Geometry'' \cite{Hartshorne}.
In \cite{ZeunerMortberg23}, the authors manage to stay constructive by using
``ringed lattices'' \cite{ConstrSchemes} instead of locally ringed spaces,
but the formalization only includes affine schemes.
The functor of points approach is often taken to be more amenable for constructive mathematics.\footnote{See e.g.\ the discussion  \url{https://github.com/agda/cubical/issues/657}}
Indeed, to our knowledge we present the first fully constructive formalization
of quasi-compact and quasi-separated schemes (qcqs-schemes), an important subclass of
schemes that is sufficient for a large portion of modern algebraic geometry.\footnote{See
  e.g.\ Deligne and Boutot's \cite{EGA4.5}, where schemes are always assumed qcqs.}

Nowadays there exist extensive algebra and category theory libraries for many
of the major proof assistants, providing a lot of the necessary tools
to formally define schemes using the functor of points approach.
The bottlenecks of defining schemes as locally ringed spaces,
disappear when following the functor of points approach.
One problem that occurs, however, is that the
category of functors from rings to sets is not locally small, since
arrows between two such functors are natural transformations,
i.e. families of functions indexed by the ``big'' type of all rings in
a given universe. As a result, one has to address size issues.
The restriction to qcqs-schemes comes from the fact that in a predicative
type theory like \CubicalAgda's, one is led to make certain finiteness assumptions
in order to deal with the arising size issues appropriately.

Our work is completely formalized in \CubicalAgda~and
all results are integrated in the \agdaCubical~library.\footnote{\label{footnote:formalization} The formalization is summarized in:\\
  \url{https://github.com/agda/cubical/blob/master/Cubical/Papers/FunctorialQcQsSchemes.agda} \\
  A clickable rendered version 
  can be found  here:\\
  \url{https://agda.github.io/cubical/Cubical.Papers.FunctorialQcQsSchemes.html}}
We will comment on our usage of \CubicalAgda~in \cref{subsec: CubicalAgda},
but we want to stress that the formalization does not rely on cubical features.
It should be possible to more or less directly translate the formalization
into a system implementing Homotopy Type Theory and Univalent Foundations
of the HoTT book\cite{HoTTBook} or into \UniMath~\cite{UniMath}.
Our work can be understood as being in line with the goals of
Voevodsky's Foundations library \cite{VoevodskyFoundationsLib}:
Developing a library of constructive set-level mathematics
based on Univalent Foundations.

As result of working fully constructive, our presentation deviates
from the standard ``Introduction to Algebraic Geometry and Algebraic Groups''
by Demazure and Gabriel \cite{DemazureGabriel}.
Our main contributions and design choices can be summarized as follows:
\begin{itemize}
\item In \cref{sec: ZFunctors} we define the category of $\Z$-functors,
  differing slightly from Demazure and Gabriel.
  Roughly speaking, we chose to work with a fully-faithful
  spectrum functor with the caveat that this functor does not have
  a left adjoint but only a relative adjoint.
\item In \cref{sec: localZFunctors} we define the notion of coverage and sheaf wrt.\ a
  coverage. We define the Zariski coverage on $\mathsf{CommRing}^{op}$.
  Restricting from $\Z$-functors to Zariski sheaves can be seen as
  introducing a locality condition, akin to restricting
  from ringed to locally ringed spaces.
  We show that affine schemes are local, i.e.\ that
  representable presheaves are sheaves wrt.\
  the Zariski coverage. For this one can reuse some key algebraic lemmas,
  first formalized in \cite{ZeunerMortberg23} to show the sheaf property of
  the structure sheaf of an affine scheme.
\item In \cref{sec: CompOpensAndSchemes} we define the notions of compact open
  subfunctor, cover of compact opens and finally qcqs-scheme.
  It is in this section that we deviate substantially from the standard sources.
  We argue that the above notions are most conveniently defined
  by using an appropriate classifier in the topos theoretic sense.
  Since we have a small Zariski lattice but no small type of radical ideals
  in \CubicalAgda, we can only classify \emph{compact opens}.
  So far, these only appear in the literature on synthetic algebraic geometry
  (\cite[Def.\ 19.15]{BlechschmidtPhD} and \cite[Def.\ 4.2.1]{SAG}),
  but they turn out to be very useful for our purposes as well.
\item In \cref{sec: openSubschemes} we prove that compact open subfunctors
  of affine schemes are qcqs-schemes. We give a point-free proof that the classifier
  for compact opens is separated, only using the universal property of the Zariski
  lattice. This gives us that compact opens of affine schemes are sheaves.
  The fact that compact opens of affines have an affine cover essentially follows from
  the Yoneda lemma.
\end{itemize}

%% file: background.tex
\section{Background}\label{sec: background}
We begin by giving some helpful background.
First, we discuss the \CubicalAgda~proof assistant and how it is used in the formalization.
We then briefly present two algebraic constructions from the \agdaCubical~library,
first formalized and described by Zeuner and Mörtberg in \cite{ZeunerMortberg23},
that play a key role in this paper as well:
localizations of commutative rings and the Zariski lattice.

\subsection{Univalent type theory in \CubicalAgda}\label{subsec: CubicalAgda}
For understanding the details of our formalization, it is worth knowing about
certain particularities of the \CubicalAgda~proof assistant and its
library. We will restrict ourselves to the features that are relevant for this paper.
Readers familiar with \CubicalAgda~or Homotopy Type Theory and
Univalent Foundations (HoTT/UF) can safely skim this section.
Readers interested in more details are referred to \cite{CubicalAgda2}.

\CubicalAgda~is a rather recent extension of the \Agda~proof assistant
with fully constructive support of the univalence principle
and higher inductive types (HITs). The notation used in this paper
is inspired by \Agda's syntax and the conventions of the \agdaCubical
library but we have taken the liberty to simplify the syntax
and omit projections whenever possible in order to increase readability.
For example we will write \func{CommRing} to denote both the type
and the category of commutative rings and an element $R:\func{CommRing}$
will denote both the ring with its structure and the carrier-type of $R$,
i.e.\ we write $f:R$ for its elements. For the universe at level~$\ell$
we write $\func{Type}~\ell$ or $\func{Type}_\ell$, and similarly
$\func{CommRing}_\ell$ for commutative rings whose carrier type lives in $\func{Type}_\ell$.
For a family $B:A\to\Type~\ell$, we denote the dependent pair type over this family
as $\tySigmaNoParen{x}{A}{B(x)}$.

For definitional equalities we use $=$, while
propositional equalities are written using $\func{≡}$.
Note that \CubicalAgda~does not use Martin-Löf's inductive identity type
\cite{MartinLof75itt} for expressing propositional equalities,
but rather so-called \emph{path} types. These path types are defined in
terms of a primitive interval type \func{I}, which allows one to
conveniently define \emph{dependent} path types.
In this formalization we will not make direct use of the interval or
dependent path types. However, path types do entail function extensionality,
the right behavior of equalities of dependent pairs
and other useful principles, which we will use freely.\footnote{These
  principles also follow from univalence albeit with a slightly different
  computational behavior.}

\CubicalAgda~does not come with a designated universe of propositions
and in fact we cannot generally expect propositional equality types, or rather path types,
to be propositions in any sensible way.
This is because \CubicalAgda~proves univalence and thus disproves
\emph{Uniqueness of Identity Proofs} (UIP),
also known as Streicher's axiom K \cite{Streicher93}.
We can, however, internally define (proof-relevant) propositions as subsingleton types
and sets as types whose equalities are propositions, i.e.\ as types satisfying UIP:
\vspace{-\abovedisplayskip}
\begin{center}
\begin{minipage}[t]{1.0\linewidth}
\begin{minipage}{.4\linewidth}
  \ExecuteMetaData[agda/latex/Background.tex]{prop}
\end{minipage}
\begin{minipage}{.5\linewidth}
  \ExecuteMetaData[agda/latex/Background.tex]{set}
\end{minipage}
\end{minipage}
\end{center}
The type (or universe) of propositions of level $\ell$ is defined as
$\func{hProp}~\ell = \tySigma{A}{\Type~\ell}{\func{isProp}~A}$,
A \emph{subset} of $A$, where $\func{isSet}~A$, is a function
$S:A\to\hProp~\ell$. With some abuse of notation we
will identify a subset $S$ with the corresponding $\Sigma$-type
$\tySigma{a}{A}{a\in S}$,
where $a\in S$ is the proposition (type of proofs) that $a$ is actually in $S$.
We thus write $a:S$ for elements
of $S$ when the proof of $a:A$ belonging to $S$ can be ignored.

Univalence implies that there are types, which are neither propositions nor sets.
These types are said to have a higher h-level (homotopy level \cite{Voevodsky10bonn})
than sets. One can use the so-called \emph{structure identity principle} \cite[Sec.\ 9.8]{HoTTBook}
to prove that this holds true for types of algebraic or categorical structures like
commutative rings or $\Z$-functors.\footnote{See \cite{POPLPaper} for an implementation of
the structure identity principle in \CubicalAgda.}
However, we want to stress that this does not affect
the formalization presented in this paper.

We do make extensive use higher inductive types (HITs), the other main
addition of HoTT/UF to dependent type theory alongside univalence.
In particular, we require two HITs: set-quotients and propositional truncations.
Set-quotients are needed to define localizations of rings and the Zariski lattice,
which we will describe in \cref{subsec: ZariskiLattice}. We will not go into details
on how set-quotients are defined. It suffices to know that as long as we quotient
sets by proposition-valued equivalence relations and only consider maps
from those quotients into other sets, everything works as one would expect from quotients.
The other HIT, propositional truncation, turns any type into a proposition:
\ExecuteMetaData[agda/latex/Background.tex]{propTrunc}
This is needed in HoTT/UF to express existential quantification,
as using $\Sigma$-types is often too strong. We follow the convention
and say ``there \emph{merely} exists $x:A$ such that $P(x)$'', if we have an inhabitant of
\begin{align*}
  \tyExistsNoParen{x}{A}{P(x)} = \ptrunc{\tySigmaNoParen{x}{A}{P(x)}}
\end{align*}
Note that in general this does not let us extract a witness $x:A$.
We will discuss an example showcasing the proper use of propositional
truncation in \cref{rem: affineCover}.

\subsection{Localizations and the Zariski lattice}\label{subsec: ZariskiLattice}
Our formalization builds on a lot of commutative algebra and category theory formalized
in the \agdaCubical~library that we will presuppose in this paper.
In particular, we will assume familiarity with presheaves, the Yoneda lemma and basic
ideal theory of rings and we will not comment on their implementation in the \agdaCubical library.
There are two particular constructions, first described in \cite{ZeunerMortberg23},
that are of special importance to this project and we will briefly describe them here.

The first, localizations of commutative rings, are a way of making elements invertible
by adding fractions.
In this paper we only need the special case of inverting a single element.
For a  ring $R$ and $f:R$ the \emph{localization of} $R$ \emph{away from} $f$ is the ring
$\locEl{R}{f}$ of fractions $\nicefrac{r}{f^n}$ where the denominator is a power of $f$.
Equality of two fractions is slightly different than for fractions of integers
and can be stated as:\footnote{This is to account for zero-divisors and the case where
  $f$ is nilpotent.}
\begin{align*}
  \tyPath{\nicefrac{r}{f^n}}{\nicefrac{r'}{f^m}} \quad \text{iff} \quad \tyExists{k}{\func{ℕ}}{\tyPath{rf^{k+m}}{r'f^{k+n}}}
\end{align*}
Localizations satisfy a universal property and in our special case it can be stated as:
$\locEl{R}{f}$ is the initial $R$-algebra where $f$ becomes invertible. This means that
for any ring $A$ with a homomorphism $\varphi:\mathsf{Hom}(R,A)$ such that
$\varphi(f)\in A^\times$ (i.e.\ $\varphi(f)$ is a unit/invertible),
there is a unique $\psi:\mathsf{Hom}(\locEl{R}{f},A)$ making
the following diagram commute
\[
\begin{tikzcd}
  &R\arrow[ld,"\nicefrac{\_}{1}"']\arrow[rd,"\varphi"]&\\
  \locEl{R}{f}\arrow[rr,dashed,"\exists!~\psi",swap]&&A
\end{tikzcd}
\]
where $\nicefrac{\_}{1}:\mathsf{Hom}(R,\locEl{R}{f})$ is the canonical morphism
mapping $r:R$ to the fraction $\nicefrac{r}{1}$.
As shown in \cite{ZeunerMortberg23}, formalizing localizations with the help of
set-quotients is straightforward.

The second construction, the Zariski lattice
associated to a ring is slightly more delicate. By a standard argument in classical
algebraic geometry there is a one-to-one correspondence between Zariski open sets
of $\mathsf{Spec}(R)$ and radical ideals of $R$. An ideal $I\subseteq R$
is radical if $I=\sqrt{I}=\{\, x\in R \;\vert\; \exists n>0:x^n\in I\,\}$.
Furthermore, the \emph{compact open} subsets of $\mathsf{Spec}(R)$ correspond
radicals of \emph{finitely generated} ideals. This correspondence is in fact
an isomorphism of lattices. Set-theoretic union and intersection of compact opens
correspond to addition and multiplication of finitely generated ideals.

This means that we can define this so-called \emph{Zariski lattice} $\ZL$ without having to define
$\mathsf{Spec}(R)$ and its topology first: Elements of $\ZL$ are generators $f_1,...,f_n:R$
quotiented by the relation that relates another list of generators $g_1,...,g_m:R$ if
$\tyPath{\sqrt{\langle f_1,\dots,f_n\rangle}}{\sqrt{\langle g_1,\dots,g_m\rangle}}$.
The equivalence class of the generators $f_1,...,f_n:R$ is denoted by $D(f_1,...,f_n):\ZL$
and the join on $\ZL$ is given by $D(f_1,...,f_n)\vee D(g_1,...,g_m)=D(f_1,...,f_n,g_1,...,g_m)$.
The ``basic open'' $D(f)$ is the equivalence class corresponding
to the radical of the principle ideal $\sqrt{\langle f\rangle}$, with $D(1)$ being
the top element of $\ZL$ corresponding to the 1-ideal.
The basic opens form a basis of $\ZL$, as $D(f_1,...,f_n)=\bigvee_{i=1}^n D(f_i)$.

This definition is due to Espa\~{n}ol \cite{Espanol83}, but it has the disadvantage that
it uses equality of ideals to define the quotienting relation. In the predicative type
theory of \CubicalAgda~the type of ideals of $R$ lives in the
universe above $R$ an so does the equality type
between two ideals. This can be avoided by slightly rewriting
the equivalence relation, as shown in \cite{ZeunerMortberg23},
giving us $\ZL:\func{DistLattice}_\ell$ for $R:\func{CommRing}_\ell$.

Joyal \cite{JoyalZarLat} observed that the Zariski lattice has a certain universal property
that can be stated in terms of \emph{supports}.
A map $d:R\to L$ from $R$ into a distributive lattice $L$ is called a support if it satisfies:
\begin{align}
  & \tyPath{d(1)}{\top} \;\text{ and }\; \tyPath{d(0)}{\bot} \\
  & \forall (f~g :R)\to \tyPath{d(fg)}{d(f)\wedge d(g)} \\
  & \forall (f~g :R)\to d(f+g)\leq d(f)\vee d(g)
\end{align}
The map $D:R\to\ZL$ sending $f:R$ to the equivalence class $D(f)$ satisfies conditions (1)-(3)
and it is a universal support in the sense for any other support $d:R\to L$
there is a unique lattice homomorphism $\varphi:\ZL\to L$ such that
the following commutes
\[
\begin{tikzcd}
  & R \arrow[dl,"D"']\arrow[dr,"d"] & \\
  \ZL \arrow[rr,dashed, "\exists!~\varphi"'] && L
\end{tikzcd}
\]
The partial order defined on the Zariski lattice is connected to localizations as for $f,g:R$
\begin{align*}
  D(g)\leq D(f)\;\Leftrightarrow\;\sqrt{\langle g\rangle}\subseteq\sqrt{\langle f\rangle}\;\Leftrightarrow\;g\in\sqrt{\langle f\rangle} \;\Leftrightarrow\; \nicefrac{f}{1}\in\locEl{R}{g}^\times
\end{align*}
In the special case where $g=1$, this gives us $\tyPath{D(1)}{D(f)}$ iff $f\in R^\times$.
We will utilize this fact in order to interpret the basic opens as affine subschemes
in \cref{def: standardOpen}. A slight generalization of this fact
that we will exploit repeatedly is that
for $f_1,...,f_n:R$
\begin{align*}
  \tyPath{D(1)}{D(f_1,...,f_n)} \;\Leftrightarrow\; 1\in\langle f_1,...,f_n\rangle
\end{align*}
This concludes our discussion of the preliminaries required to formalize
qcqs-schemes following the functor of points approach.

%% file: zfunctors.tex
\section{$\Z$-Functors}\label{sec: ZFunctors}
Let us turn to our goal of defining qcqs-schemes
as well-behaved functors from rings to sets.
As size issues are unavoidable in the functor of points approach,
we will be rather explicit about universe levels in this paper.
For the remainder we will fix a universe level $\ell$ and work
over commutative rings in the corresponding universe $\func{CommRing}_\ell$.

\begin{definition}\label{def: ZFunctors}
  The category of $\Z$-functors, denoted $\func{ℤFunctor}_\ell$,
  is the category of functors from $\func{CommRing}_\ell$ to $\func{Set}_\ell$.
  We write $\func{Sp}:\func{CommRing}_\ell^{op}\to\func{ℤFunctor}_\ell$
  for the Yoneda embedding and $\Aone :\func{ℤFunctor}_\ell$ for the
  forgetful functor from commutative rings to sets.
  We say that $X:\ZFunctor_\ell$ is an \emph{affine scheme}
  if there merely exists $R:\func{CommRing}_\ell$ such that $X\cong\func{Sp}(R)$.
\end{definition}

\begin{remark}\label{rem: diffDemazureGabriel}
  It is worth noticing that most modern algebraic geometry sources
  (see e.g.\ \cite{EisenbudHarris,GoertzWedhorn,NPNotes,RisingSea}) usually
  omit any reference to universes when discussing the functor of points approach.
  The choice of taking functors from rings to sets in the same universe
  seems perhaps most natural, but actually differs from the standard
  reference on the functor of points approach by Demazure and Gabriel \cite{DemazureGabriel}.
  They essentially take $\Z$-functors to be functors from
  $\func{CommRing}_\ell$ to $\func{Set}_{\ell+1}$.\footnote{They actually assume two
    Grothendieck universes $\mathcal{U}\subseteq\mathcal{V}$.
    As type theoretic universes are usually ``lifted''
    from Grothendieck universes in presheaf models \cite{HofmannStreicher1997},
    our translation only seems natural.}
  Their ``big spectrum functor''
  $\mathsf{Sp}:\func{CommRing}_{\ell+1}^{op}\to(\func{CommRing}_\ell\to\func{Set}_{\ell+1})$
  is defined much like the Yoneda embedding as $\mathsf{Sp}(R)=\mathsf{Hom}(R,\_)$,
  but because of the universe level mismatch it is \emph{not}
  fully faithful. However, this functor has a left adjoint,
  namely the functor that we will define in \cref{def: globalSectionsFun}.
  We decided to differ in our definition of $\Z$-functors
  since \Agda's non-cumulative universes would require explicit lifts
  in a lot of places and it seemed more useful to a have \func{Sp}
  that is fully-faithful.
\end{remark}


\begin{definition}\label{def: globalSectionsFun}
  Let $X:\ZFunctor_\ell$, the \emph{ring of functions}
  $\Ofun(X)$ is the type of natural transformations $\NatTrans{X}{\Aone}$
  equipped with the canonical point-wise operations,
  i.e.\ for $R:\func{CommRing}_\ell$ and $x:X(R)$, addition and multiplication
  of $\alpha,\beta:\NatTrans{X}{\Aone}$ are given by
  \begin{align*}
    (\alpha + \beta)_R(x) ~=~ \alpha_R(x) + \beta_R(x) \quad\quad
    (\alpha \cdot \beta)_R(x) ~=~ \alpha_R(x) \cdot \beta_R(x)
  \end{align*}
  This defines a functor $\Ofun:\ZFunctor_\ell\to\func{CommRing}_{\ell+1}^{op}$,
  whose action on morphisms (natural transformations) is given by precomposition.
\end{definition}
The universal property of schemes is often stated to be:
The global sections functor $\Gamma$ is left adjoint to $\mathsf{Spec}$
and the counit of this adjunction is an isomorphism.
However, this is already true for locally ringed spaces. In a similar fashion
we would like to have an adjunction $\Ofun\dashv\func{Sp}$,
but unfortunately we run into a universe level mismatch.
We still get something that looks a lot like an adjunction.
The proof of the following proposition is straightforward.
\begin{proposition}\label{prop: OSpAdj}
  For $R:\func{CommRing}_\ell$ and $X:\ZFunctor_\ell$ there is an isomorphism
  of types
  \begin{align*}
    {\mathsf{Hom}\big(R,\Ofun(X)\big)}\cong\big(\NatTrans{X}{\func{Sp}(R)}\big)
  \end{align*}
  which is natural in both $R$ and $X$. Moreover, the induced ``counit''
  $\varepsilon_R:\mathsf{Hom}\big(R,\Ofun(\func{Sp}(R))\big)$,
  which is obtained by applying the inverse of above isomorphism
  to the identity transformation $\NatTrans{\func{Sp}(R)}{\func{Sp}(R)}$,
  is an isomorphisms of rings for all $R:\func{CommRing}_\ell$.
\end{proposition}

\begin{remark}\label{rem: relAdj}
  \cref{prop: OSpAdj} type-checks because the type of ring homomorphisms
  is universe polymorphic, meaning it can take rings living in different
  universes as arguments. The same holds for the type of isomorphisms/equivalences
  between two types.   From a categorical perspective, we get
  a \emph{right relative adjunction}\footnote{See
    \url{https://ncatlab.org/nlab/show/relative+adjoint+functor}}
  $\Ofun\dashv_{\,\,i}\func{Sp}$ with
  respect to the inclusion, or lift functor,
  $i:\func{CommRing}_\ell^{op}\to\func{CommRing}_{\ell+1}^{op}$.
  This is why we only get a counit, but no unit.
\end{remark}

%% file: coverage.tex
\section{Local $\Z$-functors}\label{sec: localZFunctors}
Functorial (qcqs-) schemes are sheaves with respect to the Zariski coverage.
The notion of coverage (also called a Grothendieck pre-topology) generalizes
point-set topologies to arbitrary categories.
Roughly speaking, a coverage on a category $\mathcal{C}$
associates to each object $U:\mathcal{C}$
a family $\mathsf{Cov}(U)$ of covers. A cover $(U_i\to U)_{i: I}:\mathsf{Cov}(U)$
is a family of maps into $U$. These families $\mathsf{Cov}(U)$ should satisfy certain
closure properties. If $\mathcal{C}$ has pullbacks
then covers should be closed under pullbacks and a presheaf
$\mathcal{F}:\mathcal{C}^{op}\to\mathsf{Set}$ can be defined to be a sheaf
if for any $(U_i\to U)_{i:I}: \mathsf{Cov}(U)$ we get and equalizer diagram
\begin{align*}
  \mathcal{F}(U) \to \prod_{i: I}\mathcal{F}(U_i)\rightrightarrows
                                              \prod_{i,j: I}\mathcal{F}(U_i \times_U U_j)
\end{align*}
In the case where $\mathcal{C}$ is $\mathsf{Open}(X)$,
the poset of open subsets of a topological space $X$,
we get a canonical coverage:
A family of opens $(U_i\subseteq U)_{i\in I}$ is in $\mathsf{Cov}(U)$
if and only if $\bigcup_{i\in I} U_i = U$. Pullbacks in $\mathsf{Open}(X)$
are given by set-theoretic intersection $\cap$ and we recover the usual
definition of when a presheaf $\mathcal{F}:\mathsf{Open}(X)^{op}\to\mathsf{Set}$
is a sheaf.

The formalization of coverages and sheaves in the \agdaCubical library
follows the nLab\footnote{See \url{https://ncatlab.org/nlab/show/coverage}}
and Johnstone's classic ``Sketches of an elephant'' \cite[C2]{JohnstoneSketchesVolII}.
The advantage of this approach is that it even works for categories without pullbacks.
As it turns out, it also lets us conveniently
define the Zariski coverage and prove that representables are Zariski sheaves.
For now, let us fix an arbitrary category $\mathcal{C}$.

\begin{definition}\label{def: Coverage}
  A \emph{cover} on an object $c:\mathcal{C}$ consists of an index type $I$
  and for each $i:I$ an element in the slice category $\nicefrac{\mathcal{C}}{c}$,
  i.e.\ an arrow $f_i:\mathcal{C}(c_i,c)$.
  A \emph{coverage} on $\mathcal{C}$ consist of a family of covers for each
  $c:\mathcal{C}$ satisfying \emph{pullback stability}:
  Given a cover
  $\{\,f_i:\mathcal{C}(c_i, c)\,\}_{i:I}$ of $c$ and arrow $f:\mathcal{C}(d,c)$,
  there merely exists a cover $\{\,g_j:\mathcal{C}(d_j, d)\,\}_{j:J}$ of $d$ such that
  for each index $j:J$ there merely exists an index $i:I$ and an arrow
  $h_{ij}:\mathcal{C}(d_j,c_i)$ with $\tyPath{f_i\circ h_{ij}}{f\circ g_j}$.
\end{definition}
%
Pullback stability can also be stated as: Given an arrow
$f:\mathcal{C}(d,c)$ and a cover on $c$, we can take the \emph{sieve} generated
by this cover and pull it back to a sieve on $d$. Then there exists a cover
on $d$ refining the pulled back sieve on $d$. Since sieves
are not required for the remainder of the paper, we decided to unfold
the definition of pullback stability and state it without recourse to sieves.
We refer the interested reader to the formalization.
We now define what it means to be sheaf with respect to a fixed coverage on $\mathcal{C}$.

\begin{definition}\label{def: compFam}
  Let $P$ be a presheaf on $\mathcal{C}$.
  Let $c:\mathcal{C}$ and $\{\,f_i:\mathcal{C}(c_i, c)\,\}_{i:I}$ be a cover.
  A \emph{compatible family}\footnote{Also called a matching family:
    \url{https://ncatlab.org/nlab/show/matching+family}}
  is a dependent function $x:(i:I)\to P(c_i)$, i.e.\ a family of elements $x_i:P(c_i)$,
  such that for each pair of indices $i,j:I$ and arrows $g_i:\mathcal{C}(d,c_i)$
  and $g_j:\mathcal{C}(d,c_j)$ with $\tyPath{f_j g_j}{f_i g_i}$, we have
  $\tyPath{\rest{x_j}{g_j}}{\rest{x_i}{g_i}}$ (in $P(d)$).
  We denote the type of compatible families over a cover $\{\,f_i:\mathcal{C}(c_i, c)\,\}_{i:I}$
  by \normalfont{$\func{CompatibleFam}^P\big(\{\,f_i:\mathcal{C}(c_i, c)\,\}_{i:I}\big)$}.
\end{definition}
For an element $x:P(c)$ we get an induced compatible family by taking
the restrictions $x_i=\rest{x}{f_i}$ for $i:I$. The compatibility follows
directly from the presheaf property of $P$.
This constructions gives us a map
$\sigma_P:P(c) \to \func{CompatibleFam}^P\big(\{\,f_i:\mathcal{C}(c_i, c)\,\}_{i:I}\big)$.
We can now conveniently define sheaves in terms of the map $\sigma$.

\begin{definition}\label{def: isSheaf}
  A presheaf $P$ is a \emph{sheaf} if for all $c:\mathcal{C}$ and covers
  $\{\,f_i:\mathcal{C}(c_i, c)\,\}_{i:I}$, the canonical map $\sigma_P$
  is an isomorphism.
\end{definition}

\begin{definition}\label{def: subcanonical}
  A coverage on $\mathcal{C}$ is called \emph{subcanonical}
  if for all $c:\mathcal{C}$ the Yoneda embedding of $c$ is a sheaf
  with respect to the coverage.
\end{definition}
In this paper we are interested in a particular example of a coverage on
the opposite category of commutative rings.
Covers of a ring $R$ will come from finite lists of generators of the $1$-ideal.
Classically, this corresponds to the fact that any open cover of an affine scheme
is of the form $\mathsf{Spec}(R)=\bigcup_{i=1}^n D(f_i)$ with
$1\in\langle f_1,...,f_n\rangle$ (because $\mathsf{Spec}(R)$ is quasi-compact).
We call a finite list of elements $f_1,...,f_n:R$ such that $1\in\langle f_1,...,f_n\rangle$
a \emph{unimodular vector}.

\begin{definition}\label{def: ZariskiCoverage}
  The \emph{Zariski coverage} on $\mathcal{C}=\func{CommRing}_\ell^{op}$
  is given by:
  \begin{itemize}
  \item For each $R:\func{CommRing}_\ell$, covers are indexed by the type
    of unimodular vectors over $R$.
  \item For each unimodular vector $f_1,...,f_n:R$, the associated cover of $R$
    is given by the reversed canonical morphisms $\nicefrac{\_}{1}:\locEl{R}{f_i}\to R$,
    indexed by $i:\func{Fin}~n$, the finite $n$-element type.
  \item For a unimodular vector $f_1,...,f_n:R$ the pullback along a morphism
    $\varphi:\mathsf{Hom}(R,A)$ is the vector $\varphi(f_1),...,\varphi(f_n):A$,
    which is easily shown to be unimodular as well.
  \end{itemize}
  A presheaf $X:\ZFunctor_\ell$ is called \emph{local} if it is a sheaf
  wrt.\ the Zariski coverage.
\end{definition}

\begin{lemma}\label{lem: pullbackStability}
  The Zariski coverage is stable under pullbacks.
\end{lemma}

\begin{proof}
  Let $R, A:\func{CommRing}_\ell$, $f_1,...,f_n:R$ be a unimodular vector and
  $\varphi:\mathsf{Hom}(R,A)$.
  The universal property of localization induces
  ring morphisms $\psi_i:\mathsf{Hom}\big(\locEl{R}{f_i},\locEl{A}{\varphi(f_i)}\big)$
  such that the following diagram commutes (in $\func{CommRing}_\ell^{op}$)
  \[\begin{tikzcd}
	\locEl{A}{\varphi(f_i)} & \locEl{R}{f_i} \\
	A & R
	\arrow["\nicefrac{\_}{1}", from=1-1, to=2-1]
	\arrow["\nicefrac{\_}{1}", from=1-2, to=2-2]
	\arrow["\psi_i", from=1-1, to=1-2]
	\arrow["{\varphi}", from=2-1, to=2-2]
  \end{tikzcd}\]
\end{proof}
The key result of this section uses an algebraic fact that can be found in many textbooks,
such as \cite[p. 125]{SheavesInGeometryAndLogic},
and was already formalized in \CubicalAgda~to prove \cite[Lemma 15]{ZeunerMortberg23}.

\begin{theorem}\label{thm: subcanonicalZariskiCoverage}
  The Zariski coverage is subcanonical, i.e.\ $\func{Sp}(A)$ is local for $A:\func{CommRing}_\ell$.
\end{theorem}

\begin{proof}
  Let $R:\func{CommRing}_\ell$ and $f_1,...,f_n:R$ a unimodular vector be given.
  For $i,j$ in $1,...,n$, we denote by $\chi^l_{ij}:\locEl{R}{f_i}\to\locEl{R}{f_if_j}$
  and $\chi^r_{ij}:\locEl{R}{f_j}\to\locEl{R}{f_if_j}$ the canonical morphisms
  given by the universal property of localization.
  We use without proof that the map
  \begin{align*}
    R ~\to~\tySigmaNoParen{x}{(i:\func{Fin}~n)\to\locEl{R}{f_i}}{~\forall~i~j\to\tyPath{\chi^l_{ij}(x_i)}{\chi^r_{ij}(x_j)}}
  \end{align*}
  sending $g:R$ to $\nicefrac{g}{1}:\locEl{R}{f_i}$ for $i=1,...,n$,
  is an isomorphism.
  Using this, one can construct a chain of isomorphisms
  \begin{align*}
    \mathsf{Hom}(A,R)
      ~&\cong~\tySigmaNoParen{\varphi}{(i:\func{Fin}~n)\to\mathsf{Hom}(A,\locEl{R}{f_i})}{~\forall~i~j\to\tyPath{\chi^l_{ij}\circ\varphi_i}{\chi^r_{ij}\circ\varphi_j}} \\
      ~&\cong~\func{CompatibleFam}^{\func{Sp}(A)}\big(\{\,f_i\}_{i=1,...,n}\big)
  \end{align*}
  which factors through the canonical map $\sigma_{\func{Sp}(A)}$.
\end{proof}

%% file: compactopens.tex
\section{Compact opens and qcqs-schemes}\label{sec: CompOpensAndSchemes}
The standard way to define open subfunctors follows a two step process.
First, one defines them for representables using (radical) ideals.
Then, one defines open subfunctors of general $\Z$-functors
by pulling back to representables.
Working predicatively in \CubicalAgda, we need to restrict ourselves
to finitely generated ideals, which gives us compact open subfunctors.
Let us sketch the idea behind compact opens informally to see
why this restriction is necessary:
For a f.g.\ ideal $I\subseteq A$,
we get the \emph{affine compact open subfunctor} $\mathsf{Sp}(A)_I\hookrightarrow \mathsf{Sp}(A)$ given by
\begin{align*}
  \mathsf{Sp}(A)_I(B) = \{ \varphi:\mathsf{Hom}(A,B) ~\vert~ \varphi^*I=B \} \subseteq \mathsf{Sp}(A)(B)
\end{align*}
If $I=\langle f_1,...,f_n \rangle$, then the ``pullback'' along  $\varphi:\mathsf{Hom}(A,B)$
is just $\varphi^*I=\langle \varphi(f_1),...,\varphi(f_n) \rangle$.
With this, we can define a subfunctor $U\hookrightarrow X$ to be \emph{compact open} if
pulling back along an $A$-valued point of $X$ gives an affine compact open subfunctor
of $\mathsf{Sp}(A)$, i.e.\ if for any ring $A$
and $\phi:{\mathsf{Sp}(A)}\Rightarrow {X}$ there is a f.g.\ ideal $I \subseteq A$ such that
the following is a pullback square
\[\begin{tikzcd}
	\mathsf{Sp}(A)_{I} & U \\
	\mathsf{Sp}(A) & X
	\arrow[hook, from=1-1, to=2-1]
	\arrow[hook, from=1-2, to=2-2]
	\arrow["\lrcorner"{anchor=center, pos=0.125}, draw=none, from=1-1, to=2-2]
	\arrow[from=1-1, to=1-2]
	\arrow["{\phi}", from=2-1, to=2-2]
\end{tikzcd}\]
Note that the ideal $I$ is not uniquely determined in this case.
Indeed, if $I=\langle f_1,...,f_n \rangle$ and $J=\langle g_1,...,g_m \rangle$
are such that ${\sqrt{I}}={\sqrt{J}}$, then
for any $\varphi:\mathsf{Hom}(A,B)$ we have
\begin{align*}
  1\in\langle \varphi(f_1),...,\varphi(f_n) \rangle \quad\text{iff}\quad
  1\in\langle \varphi(g_1),...,\varphi(g_m) \rangle
\end{align*}
and thus $\mathsf{Sp}(A)_I\cong\mathsf{Sp}(A)_J$.
In fact, one can prove that the converse also holds.
This means that the compact open $U$ and the $A$-valued point
$\phi:{\mathsf{Sp}(A)}\Rightarrow {X}$
determine a finitely generated ideal $I=\langle f_1,...,f_n \rangle$
up to equality of radical ideals,
i.e.\ an element  $D(f_1,...,f_n)$ of the Zariski lattice $\ZarLat{A}$.
Note that we can describe $\mathsf{Sp}(A)_I$ purely in terms of
$D(f_1,...,f_n)$, as the $B$-valued points are given by
\begin{align*}
  \mathsf{Sp}(A)_I(B) &= \{ \varphi:\mathsf{Hom}(A,B) ~\vert~ 1\in\langle \varphi(f_1),...,\varphi(f_n) \rangle \} \\
  &= \{ \varphi:\mathsf{Hom}(A,B) ~\vert~ D\big(\varphi(f_1),...,\varphi(f_n)\big)=D(1) \rangle \}
\end{align*}
The pullback condition ensures that this mapping is natural in $A$. In other words,
the compact open subfunctors of $X$ are in one-to-one correspondence
with natural transformations from $X$ to the $\Z$-functor $\mathcal{L}$
that sends a ring to its Zariski lattice. Note that we can define
this $\Z$-functor $\mathcal{L}$ because of the ``small'' definition of
Zariski lattice. If we drop the finiteness assumption on ideals to get open
subfunctors we cannot hope to define the classifier in \CubicalAgda.
We will discuss possibilities to do so in other systems in
\cref{subsec: NonConstr}.

For a topos theorist it might not constitute a particularly deep insight
that the compact open subfunctors (sub-objects) of a $\Z$-functor are \emph{classified}
by the ``internal Zariski lattice'' $\mathcal{L}$. This means that the compact opens
are precisely given by pullbacks of the form
\[\begin{tikzcd}
	U & \mathbf{1} \\
	X & \mathcal{L}
	\arrow[hook, from=1-1, to=2-1]
	\arrow["D(1)", from=1-2, to=2-2]
	\arrow["\lrcorner"{anchor=center, pos=0.125}, draw=none, from=1-1, to=2-2]
	\arrow[from=1-1, to=1-2]
	\arrow[from=2-1, to=2-2]
\end{tikzcd}\]
where $D(1):{\mathbf{1}}\Rightarrow {\mathcal{L}}$ is the ``constant'' natural transformation,
sending the point of the terminal $\Z$-functor $\mathbf{1}$ to the top
element of the Zariski lattice. From a formal perspective however,
we found it significantly more convenient to work with natural transformations
into $\mathcal{L}$ and the induced subfunctors, as opposed to following
the text-book strategy of defining compact-openness as a property of subfunctors
through the two step process outlined above.\footnote{A rare exception to following the
  standard definition is this blog-post \cite{MadoreDefScheme}.}
We will thus proceed to describe how compact opens can be formally defined as natural
transformations and how this gives a concise definition of qcqs-schemes.

\begin{definition}\label{def: ZarLatFun}
  Let $\mathcal{L}:\ZFunctor_\ell$ be the $\Z$-functor mapping a ring
  $R:\func{CommRing}_\ell$ to the underlying set of the Zariski lattice
  $\ZL$. The action on morphisms is induced by the universal property of the
  Zariski lattice, i.e.\ for $\varphi:\mathsf{Hom}(A,B)$ we take
  \[
  \begin{tikzcd}
    & A \arrow[dl,"D"']\arrow[dr,"D \circ\varphi"] & \\
    \mathcal{L}_A \arrow[rr,dashed, "\exists!~\ZarLat{\varphi}"'] && \mathcal{L}_B
  \end{tikzcd}
  \]
\end{definition}

\begin{definition}\label{def: CompOpens}
  Let $X:\ZFunctor_\ell$, a \emph{compact open} of $X$ is a natural transformation
  $U:\NatTrans{X}{\mathcal{L}}$. The \emph{realization} $\coBrackets{U}:\ZFunctor_\ell$
  of a compact open $U$ of $X$, is given by
  \begin{normalfont}
  \begin{align*}
    \coBrackets{U}\,(R)~=~\tySigmaNoParen{x}{X(R)}{\tyPath{U(x)}{D(1)}}
  \end{align*}
  \end{normalfont}
  A compact open $U$ is called \emph{affine}, if its realization is affine, i.e.\
  if there merely exists $R:\func{CommRing}_\ell$ such that
  $\coBrackets{U}\cong\func{Sp}(R)$.
\end{definition}
The reader may verify that for $U:\NatTrans{X}{\mathcal{L}}$,
$R:\func{CommRing}_\ell$ and $x:X(R)$ such that $U(x)=D(f_1,...,f_n)$, we have
\[\begin{tikzcd}
	\func{Sp}(R)_{\langle f_1,...,f_n \rangle} & \coBrackets{U} & {\mathbf{1}} \\
	\func{Sp}(R) & X & {\mathcal{L}}
	\arrow[hook, from=1-1, to=2-1]
	\arrow["{D(1)}", from=1-3, to=2-3]
	\arrow[hook, from=1-2, to=2-2]
	\arrow[from=1-2, to=1-3]
	\arrow["U", from=2-2, to=2-3]
	\arrow["\lrcorner"{anchor=center, pos=0.125}, draw=none, from=1-2, to=2-3]
	\arrow["\lrcorner"{anchor=center, pos=0.125}, draw=none, from=1-1, to=2-2]
	\arrow[from=1-1, to=1-2]
	\arrow["{\phi_x}", from=2-1, to=2-2]
\end{tikzcd}\]
where $\phi_x$  corresponds to the $R$-valued point $x$ by the Yoneda lemma.

Since $\mathcal{L}$ is a presheaf that takes values in distributive lattices
and its restriction maps are lattice morphisms, it is an \emph{internal} lattice
in the presheaf topos of $\Z$-functors.\footnote{It is even an internal lattice
  the big Zariski topos, i.e.\ in local $\Z$-functors. However for our purposes,
  we do not need that $\mathcal{L}$ is a Zariski sheaf.}
As such, it endows the compact opens with a distributive lattice structure.

\begin{definition}\label{def: compOpenDistLat}
  Let $X:\ZFunctor_\ell$, the \emph{lattice of compact opens}
  $\func{CompOpen}(X)$ is the type $\NatTrans{X}{\mathcal{L}}$
  equipped with the canonical point-wise operations,
  i.e.\ for $R:\func{CommRing}_\ell$ and $x:X(R)$, top, bottom, join and meet are given by
  \begin{align*}
    &\top_R(x) = D(1), \quad  \bot_R(x) = D(0) \\
    &(U\wedge V)_R(x) ~=~ U_R(x) \wedge V_R(x) \\
    &(U\vee V)_R(x) ~=~ U_R(x) \vee V_R(x)
  \end{align*}
  This defines a functor $\func{CompOpen}:\ZFunctor_\ell\to\func{DistLattice}_{\ell+1}^{op}$.
\end{definition}

\begin{definition}\label{def: qcqsScheme}
  $X:\ZFunctor_\ell$ is a \emph{qcqs-scheme} if it is a local $\Z$-functor and has
  an affine cover by compact opens. That is, there merely exist compact opens
  $U_1,...,U_n:\NatTrans{X}{\mathcal{L}}$ such that each $U_i$ is affine and
  $\tyPath{\top}{\bigvee_{i=1}^n U_i}$ in the lattice $\func{CompOpen}(X)$.
\end{definition}
As an immediate sanity check we get that affine schemes are qcqs-schemes:
\begin{proposition}\label{prop: affineIsQcqsScheme}
  $\func{Sp}(R)$ is a qcqs-scheme, for $R:\func{CommRing}_\ell$.
\end{proposition}
\begin{proof}
  $\func{Sp}(R)$ is local by \cref{thm: subcanonicalZariskiCoverage}.
  The top element $\top:\func{CompOpen}(\func{Sp}(R))$
  is the ``constant'' natural transformation, sending everything to  $D(1)$, which
  by the Yoneda lemma corresponds to $D(1):\ZL$.
  It thus constitutes a trivial affine cover with $\coBrackets{\top}\cong\func{Sp}(R)$.
\end{proof}

\begin{remark}\label{rem: affineCover}
  Of course, a qcqs-scheme $X$ can and will have multiple different covers.
  \cref{def: qcqsScheme} suggests that
  ``having an affine cover'' should be expressed
  using nested mere existential quantification.
  In practice, it is more convenient to define
  the record-type \func{AffineCover}, of all affine covers of $X$,
  consisting of a finite list or vector of compact opens and proofs
  that these compact opens are affine and cover $X$.
  The property of having an affine cover is then defined as the
  truncation of this record type:
  \vspace{-\abovedisplayskip}
  \begin{center}
  \begin{minipage}[t]{1.0\linewidth}
  \begin{minipage}[t]{.48\linewidth}
    \ExecuteMetaData[agda/latex/ZFunctors.tex]{AffineCover}
  \end{minipage}
  \begin{minipage}[t]{.5\linewidth}
    \ExecuteMetaData[agda/latex/ZFunctors.tex]{hasAffineCover}
  \end{minipage}
  \end{minipage}
  \end{center}
  %
  If we want to prove a proposition about a qcqs-scheme $X$,
  we can assume we a have witness of type $\func{AffineCover}(X)$.
  If we want to map from $X$ into a set by using that $X$ has an affine
  cover, we have to show that the mapping is independent of the cover.\footnote{
    This holds by the general \emph{elimantion} principle of the propositional truncation
    due to Kraus \cite{KrausPropTrunc}.
    When mapping into types that are not sets, things get complicated very quickly.}
  This very much in line with informal mathematical practice.
\end{remark}

\begin{remark}\label{rem: inducedCover}
  One big advantage of using the internal lattice $\mathcal{L}$ to classify
  compact opens is that we get the notion of cover for free from the induced
  lattice operations in \cref{def: compOpenDistLat}.
  In textbooks, a cover by open subfunctors is usually defined directly
  using addition of ideals \cite{DemazureGabriel,NPNotes} or by taking the set-theoretic
  union at field-valued points \cite{EisenbudHarris}. The latter is not an option for
  our purposes, as the notion of field is not well-behaved constructively.
  In the \CubicalAgda~library, the join  $\_\vee\_:\ZarLat{A}\to\ZarLat{A}\to\ZarLat{A}$
  is also defined in terms of ideal addition,
  but we can upstream the necessary constructions
  and do not have to concern ourselves with pullbacks of $\Z$-functors.
\end{remark}

%% file: opensubschemes.tex
\section{Open subschemes}\label{sec: openSubschemes}
The benchmark for a workable formal definition of schemes
as locally ringed spaces, as in \cite{SchemesHOL,SchemesLean},
usually consists of a proof of the ``universal property'',
i.e.\ an adjunction $\Gamma\dashv\mathsf{Spec}$ where the counit is an isomorphism.
\cref{prop: OSpAdj}, the functorial analogue is rather straightforward
to prove. Instead, we give a proof that compact opens
of affine schemes are qcqs-schemes. We start by showing that compact opens of
Zariski sheaves are Zariski sheaves. Essentially, this holds because compact opens
are classified by $\mathcal{L}$, which is itself a Zariski sheaf. As it turns out, however,
it is sufficient to prove something weaker.

For the remainder of the paper we adopt
the following notation: For a ring $R$ and elements $f:R$ and $u:\ZL$, we
write $\rest{u}{\locEl{R}{f}}:\ZarLat{\locEl{R}{f}}$ for the result of applying
$\ZarLat{(\nicefrac{\_}{1})}$, the $\mathcal{L}$-action on the canonical
morphism. In particular we have
$\rest{D(g_1,...,g_m)}{\locEl{R}{f}}=D(\nicefrac{g_1}{1},...,\nicefrac{g_m}{1})$.

\begin{lemma}\label{lem: ZarLatSep}
  $\mathcal{L}$ is Zariski-separated, i.e.\ for $R:\func{CommRing}_\ell$ and
  $f_1 ,..., f_n:R$ unimodular the following holds:
  given $u,v:\ZL$, if $\tyPath{\rest{u}{\locEl{R}{f_i}}}{\rest{v}{\locEl{R}{f_i}}}$
  for all $i=1,...,n$, then $\tyPath{u}{v}$.
\end{lemma}

\begin{proof}
  Let $R:\func{CommRing}_\ell$ and $f_1 ,..., f_n:R$ unimodular be given together with
  $u,v:\ZL$ satisfying $\tyPath{\rest{u}{\locEl{R}{f_i}}}{\rest{v}{\locEl{R}{f_i}}}$
  for all $i=1,...,n$. Recall that for $i=1,...,n$, the restriction
  $\rest{\_}{\locEl{R}{f_i}}:\ZL\to\ZarLat{\locEl{R}{f_i}}$
  is induced by the support $D(\nicefrac{\_}{1}):R\to\ZarLat{\locEl{R}{f_i}}$.
  Now let us fix an $i=1,...,n$. Much like in classical algebraic geometry,
  we can identify $\ZarLat{\locEl{R}{f_i}}$ with $\downarrow D(f_i)$,
  the lattice of elements of $\ZL$ smaller than $D(f_i)$.\footnote{Showing that
    $\mathsf{Spec}(\locEl{R}{f})$ is homeomorphic to $D(f)$ is a standard
    exercise in algebraic geometry.}
  The map $d:\locEl{R}{f_i}\to\,{\downarrow{D(f_i)}}$ given by
  $d(\nicefrac{r}{f_i^n})= D(r)\wedge D(f_i)$ defines a support and thus induces
  a morphism $\varphi:\ZarLat{\locEl{R}{f_i}}\to\,{\downarrow D(f_i)}$.

  Now, consider the map $\_\wedge D(f_i):\ZL\to\,{\downarrow D(f_i)}$.
  We claim that $\_\wedge D(f_i)$ factors through $\varphi$.
  By the universal property of $\ZL$, there is a unique $\psi:\ZL\to\,\downarrow D(f_i)$,
  such that $\tyPath{\psi\circ D}{d(\nicefrac{\_}{1})}$.
  Both $\_\wedge D(f_i)$ and $\varphi(\rest{\_}{\locEl{R}{f_i}})$
  satisfy the same commutativity condition as  $\psi$, which implies
  $\tyPath{\_\wedge D(f_i)}{\tyPath{\psi}{\varphi(\rest{\_}{\locEl{R}{f_i}})}}$.
  Pictorially, this amounts to observing that the following diagram commutes
  \[\begin{tikzcd}
          R && {\locEl{R}{f_i}} \\
          \ZL && {\ZarLat{\locEl{R}{f_i}}} && {\downarrow D(f_i)}
          \arrow["{d}", from=1-3, to=2-5]
          \arrow["{\rest{\_~}{\locEl{R}{f_i}}}", from=2-1, to=2-3]
          \arrow["{\varphi}", from=2-3, to=2-5]
          \arrow["{\_\wedge D(f_i)}"', curve={height=18pt}, from=2-1, to=2-5]
          \arrow["D"', from=1-3, to=2-3]
          \arrow["{\nicefrac{\_}{1}}", from=1-1, to=1-3]
          \arrow["D"', from=1-1, to=2-1]
  \end{tikzcd}\]
  From our assumption it thus follows that
  \begin{align*}
    \tyPath{u\wedge D(f_i)}
           {\tyPath{\varphi(\rest{u}{\locEl{R}{f_i}})}
             {\tyPath{\varphi(\rest{v}{\locEl{R}{f_i}})}{v\wedge D(f_i)}
             }
           }
  \end{align*}
  for all $i=1,...,n$. Since the $f_i$'s are unimodular,
  we get $\tyPath{D(1)}{\bigvee_iD(f_i)}$ and hence
  \begin{align*}
    \tyPath{u}
           {\tyPath{u\wedge D(1)}
             {\tyPath{\bigvee_{i=1}^n(u\wedge D(f_i))}
               {\tyPath{\bigvee_{i=1}^n(v\wedge D(f_i))}
                 {\tyPath{v\wedge D(1)}{v}
                 }
               }
             }
           }
  \end{align*}
\end{proof}

\begin{lemma}\label{lem: isLocalCompOpenOfLocal}
  If $X:\ZFunctor_\ell$ is local, then for any compact open
  $U:\NatTrans{X}{\mathcal{L}}$ its realization $\coBrackets{U}:\ZFunctor_\ell$
  is local.
\end{lemma}

\begin{proof}
  Let $R:\func{CommRing}_\ell$ and $f_1 ,..., f_n:R$ be unimodular.
  We need to construct an inverse to the map
  \begin{align*}
    \sigma_U : \tySigmaNoParen{x}{X(R)}{\tyPath{U(x)}{D(1)}} ~\to~ \func{CompatibleFam}^{\coBrackets{U}}\big(\{f_i\}_{i=1,...,n}\big)
  \end{align*}
  For $x:X(R)$ with $\tyPath{U(x)}{D(1)}$, $\sigma_U(x)$ is the family
  of elements $\rest{x}{\locEl{R}{f_i}}$.
  It is essentially the same map as the corresponding
  \begin{align*}
    \sigma_X:X(R) ~\to~\func{CompatibleFam}^X\big(\{f_i\}_{i=1,...,n}\big)
  \end{align*}
  but it keeps track of the fact that for each $i=1,...,n$ one has
  $\tyPath{U(\rest{x}{\locEl{R}{f_i}})}{D(1)}$

  Now, any compatible family of elements
  $x_i:X(\locEl{R}{f_i})$ with ${\tyPath{U(x_i)}{D(1)}}$ can be seen
  as a compatible family on $X$ by forgetting that $\tyPath{U(x_i)}{D(1)}$.
  To this family we apply the inverse map
  \begin{align*}
    \sigma_X^{-1}:\func{CompatibleFam}^X\big(\{f_i\}_{i=1,...,n}\big) ~\to~ X(R)
  \end{align*}
  that exists since $X$ was assumed local. We claim that we can
  set $\sigma_U^{-1}( \{x_i\}_{i=1,...,n})=\sigma_X^{-1}( \{x_i\}_{i=1,...,n})$
  as $\tyPath{U(\sigma_X^{-1}( \{x_i\}_{i=1,...,n}))}{D(1)}$.
  From this it also follows immediately that $\sigma_U$ and $\sigma_U^{-1}$ are
  mutually inverse.
  To prove the claim we use \cref{lem: ZarLatSep} and the fact that
  for each $i=1,...,n$:
  \begin{align*}
     \rest{U(\sigma_X^{-1}( \{x_i\}_{i=1,...,n}))}{\locEl{R}{f_i}} ~&\func{≡}~
     U(\rest{\sigma_X^{-1}( \{x_i\}_{i=1,...,n})}{\locEl{R}{f_i}}) \\
     ~&\func{≡}~U(\sigma_X(\sigma_X^{-1}( \{x_i\}_{i=1,...,n}))_i)
     ~~\func{≡}~U(x_i)
     ~\func{≡}~D(1)
  \end{align*}
\end{proof}
It remains to prove that compact opens of affine schemes (merely) have an affine cover.
Before treating arbitrary compact opens,
we introduce the standard or basic opens of a representable $\Z$-functor
with fair bit of abuse of notation.
\begin{definition}\label{def: standardOpen}
  Let $R:\func{CommRing}_\ell$ and $f:R$, the standard open
  $D(f):\NatTrans{\func{Sp}(R)}{\mathcal{L}}$ is given by applying the Yoneda lemma
  to the basic open $D(f):\ZL$.
\end{definition}

\begin{proposition}\label{prop: isAffineStandardOpen}
  For $R:\func{CommRing}_\ell$ and $f:R$, the standard open $D(f)$ is affine.
  In particular one has a natural isomorphism
  ${\coBrackets{D(f)}}\cong{\func{Sp}\big(\locEl{R}{f}\big)}$.
\end{proposition}

\begin{proof}
  The universal properties of localization and Zariski lattice give us
  for $A$-valued points
  \begin{align*}
    \func{Sp}\big(\locEl{R}{f}\big)(A)
    &= \mathsf{Hom}(\locEl{R}{f},A) \\
    &\cong \tySigmaNoParen{\varphi}{\mathsf{Hom}(R,A)}{\varphi(f)\in A^\times} \\
    &\cong \tySigmaNoParen{\varphi}{\mathsf{Hom}(R,A)}{\tyPath{D(\varphi(f))}{D(1)}} \\
    &= \coBrackets{D(f)}(A)
  \end{align*}
  We omit the proof that this is natural in $A$.
\end{proof}

\begin{theorem}
  The realization $\coBrackets{U}$ of a compact open
  $U:\NatTrans{\func{Sp}(R)}{\mathcal{L}}$ is a qcqs-scheme.
\end{theorem}

\begin{proof}
  We get that $\coBrackets{U}$ is local from \cref{lem: isLocalCompOpenOfLocal} and
  \cref{thm: subcanonicalZariskiCoverage},
  the subcanonicity of the Zariski coverage. It remains to show that
  $\coBrackets{U}$ (merely) has an affine cover. By the Yoneda lemma, the compact open $U$ corresponds
  to and element $u:\ZL$. Every element of $\ZL$ can (merely) be expressed as a join
  of basic opens, i.e.\ we can assume $\tyPath{u}{\bigvee_i D(f_i)}$ for some
  $f_1,...,f_n:R$. Since the Yoneda lemma actually gives us an isomorphism of lattices
  between $\ZL$ and $\func{CompOpen}(\func{Sp}(R))$, we get a cover of compact opens
  $\tyPath{U}{\bigvee_i D(f_i)}$ which is affine by \cref{prop: isAffineStandardOpen}.
  Note that this is an equality in the lattice $\func{CompactOpen}(X)$.
  But since $D(f_i)\leq U$ in $\func{CompactOpen}(X)$ for $i=1,...,n$,
  we may regard the $D(f_i)$ as affine compact opens of $\coBrackets{U}$
  covering  of the top element of $\func{CompactOpen}\big(\coBrackets{U}\big)$.
\end{proof}

%% file: conclusion.tex
\section{Conclusion}\label{sec: Conclusion}
In this paper we presented a formalization of qcqs-schemes
as a full subcategory of the category of $\Z$-functors.
We defined the Zariski coverage on $\func{CommRing}_\ell^{op}$ and proved it subcanonical.
This let us define locality of $\Z$-functors and conclude that affine schemes,
i.e.\ representable $\Z$-functors, are local.
When formalizing the notion of an open covering, we introduced compact open subfunctors.
We argued that compact opens can conveniently be classified by
the $\Z$-functor that maps a ring to its Zariski lattice.
We leveraged this fact to automatically obtain a notion of covering by compact opens
and thus a formal definition of qcqs-schemes. Finally, we gave a fully constructive proof that
compact opens of affine schemes are qcqs-schemes using only point-free methods.

As mentioned before, our formalization should be regarded as a univalent
rather than a cubical formalization.
We do not depend on cubical features of \CubicalAgda~such as the interval.
However, we are adopting the univalent approach of distinguishing propositions,
sets etc.\ internally and we do require the propositional truncation and the set-quotient HITs.
Univalence is only used in the guise of its useful consequences like function extensionality.

\subsection{Going classical}\label{subsec: NonConstr}
\CubicalAgda's type theory is fully constructive and predicative.
Using set-quotients, we can define the Zariski lattice over a ring
living in the same universe as the base ring, as shown in \cite{ZeunerMortberg23}.
This predicative definition is essential for defining the classifier
$\mathcal{L}:\ZFunctor_\ell$ of compact opens and thus plays a key role in
our definition of functorial qcqs-schemes.
This makes our approach easily extensible with the use of additional logical assumptions.
If one would want to formalize not only qcqs- but general schemes
using the functor of points approach, this should be directly possible
by using a classifier for opens, not only compact opens, instead.

Assuming impredicativity, e.g. in the form
of Voevodsky's resizing axioms \cite{VoevodskyResizing},
one could define the classifier for open subfunctors as the
$\Z$-functor sending a ring $R$ to the \emph{frame} of radical ideals of $R$.\footnote{Impredicativity
  is needed to ensure that the type of ideals of a ring $R$
  lives in the same universe as $R$.}
Alternatively, assuming classical logic, one could use the frame
of Zariski-open subsets of $\mathsf{Spec}(R)$ as the classifier. This
also induces a notion of cover (not necessarily finite this time) and
hence a notion of general functorial schemes.  We expect that in this
situation one can closely follow the approach of \cref{sec: openSubschemes}
to get a corresponding proof that open subfunctors
of affine schemes are schemes. The only difference being
perhaps the proof of \cref{lem: ZarLatSep}
that the classifier is separated wrt.\ the Zariski coverage.

We decided to stick to qcqs-schemes not only because crucial
tools like the Zariski lattice were already available in
the \agdaCubical~library. We hope that the paper contains valuable
insights for constructive mathematicians interested in the foundations
of algebraic geometry, while still being usable as a blue-print for
formalizing the functor of points approach in other (possibly classical)
proof assistants.

\subsection{Synthetic algebraic geometry}\label{subsec: SAG}
The functor of points approach allows one to develop algebraic geometry synthetically.
Here, the word synthetic means ``working in the internal language of a suitable topos''.
In our case this topos is the big Zariski topos, i.e.\ the sheaf topos of local $\Z$-functors.
From the internal point of view, Zariski sheaves look like simple sets, which
can make reasoning about them easier.
The PhD thesis of Blechschmidt \cite{BlechschmidtPhD} contains an excellent introduction
to synthetic algebraic geometry for interested readers familiar with classical algebraic geometry.

This approach can even be axiomatized. Recently,
Cherubini, Coquand and Hutzler \cite{SAG} have combined
the axiomatic approach to synthetic algebraic geometry with HoTT/UF.
By adding the axioms of synthetic algebraic geometry to a dependent type theory
with univalence and HITs one can even study the cohomology of schemes
synthetically. They give a model construction in a ``higher'' Zariski topos,
where they restrict themselves to functors from \emph{finitely presented algebras} to sets
in order to avoid size issues.
Finitely presented algebras over a ring $R$
are of the form
$\nicefrac{R[x_1,...,x_n]}{\langle p_1,...,p_m\rangle}$.
For a fixed $R$, the category of f.p.\ $R$-algebras is small and one can thus use it to
develop functorial algebraic geometry without
having to worry about universe levels. Repeating the steps outlined  in this paper for f.p.\ algebras
should give rise to a truly predicative formalization of
schemes of finite presentation over $R$.


\subsection{A constructive comparison theorem}\label{subsec: ConstrCompThm}
For the working algebraic geometer, using the functor of points approach can
sometimes be advantageous, but ultimately one wants to be able to switch seamlessly
between schemes as functors from rings to sets and schemes as locally ringed spaces.
This is made possible by the so-called \emph{comparison theorem}
\cite[p. 23]{DemazureGabriel}, which establishes an equivalence of categories
between functorial and ``geometrical'' schemes. It is proved by
constructing an adjunction between $\Z$-functors
and locally ringed spaces such that unit and counit of the adjunction become
natural isomorphisms when restricted to the respective full subcategories of schemes.

Coquand, Lombardi and Schuster \cite{ConstrSchemes} give a point-free reconstruction
of geometric qcqs-schemes that is suitable for constructive study.
Instead of using locally ringed spaces, their ``spectral schemes'' are given
as distributive lattices with a sheaf of rings. The affine scheme
associated to a ring $R$ is just the Zariski lattice $\ZL$ equipped
with the usual structure sheaf. Classically, these
spectral schemes are equivalent to convential qcqs-schemes because the
topology of a qcqs-scheme is \emph{coherent} or \emph{spectral}.\footnote{
  Stone's representation theorem for distributive lattices \cite{StoneRepresentation},
  tells us that all topological information of a coherent space
  is encoded in its lattice of compact open subsets.}

Our definition of qcqs-scheme can be seen as the functorial counterpart to the
lattice-based definition of spectral scheme
due to Coquand, Lombardi and Schuster. One would hope that these
turn out to be equivalent by a \emph{constructive} comparison theorem à la Demazure and Gabriel.
Proving such a theorem requires a decent amount of novel constructive mathematics to
be developed first. One needs to introduce a point-free notion of \emph{locally}
ringed distributive lattices that contain spectral schemes
as a full subcategory and construct a suitable adjunction with $\Z$-functors.
Due to size issues, this would only be a relative adjunction as in \cref{rem: relAdj}.
This could be particularly interesting problem in a univalent setting.

%% file: paper.bbl
\begin{thebibliography}{10}

\bibitem{POPLPaper}
Carlo Angiuli, Evan Cavallo, Anders M\"{o}rtberg, and Max Zeuner.
\newblock Internalizing representation independence with univalence.
\newblock {\em Proc. ACM Program. Lang.}, 5(POPL), January 2021.
\newblock \href {https://doi.org/10.1145/3434293} {\path{doi:10.1145/3434293}}.

\bibitem{BlechschmidtPhD}
Ingo Blechschmidt.
\newblock Using the internal language of toposes in algebraic geometry, 2021.
\newblock \href {http://arxiv.org/abs/2111.03685} {\path{arXiv:2111.03685}}.

\bibitem{SchemesHOL}
Anthony Bordg, Lawrence Paulson, and Wenda Li.
\newblock {Simple Type Theory is not too Simple: Grothendieck’s Schemes
  Without Dependent Types}.
\newblock {\em Experimental Mathematics}, 0(0):1--19, 2022.
\newblock \href {https://doi.org/10.1080/10586458.2022.2062073}
  {\path{doi:10.1080/10586458.2022.2062073}}.

\bibitem{SchemesLean}
Kevin Buzzard, Chris Hughes, Kenny Lau, Amelia Livingston, Ramon~Fernández
  Mir, and Scott Morrison.
\newblock Schemes in lean.
\newblock {\em Experimental Mathematics}, 0(0):1--9, 2021.
\newblock \href {https://doi.org/10.1080/10586458.2021.1983489}
  {\path{doi:10.1080/10586458.2021.1983489}}.

\bibitem{SchemesUniMath}
Tim Cherganov.
\newblock {Sheaf of rings on Spec R}, 2022.
\newblock URL:
  \url{https://github.com/UniMath/UniMath/blob/0df0949b951e198c461e16866107a239c8bc0a1e/UniMath/AlgebraicGeometry/Spec.v}.

\bibitem{SAG}
Felix Cherubini, Thierry Coquand, and Matthias Hutzler.
\newblock A foundation for synthetic algebraic geometry, 2023.
\newblock \href {http://arxiv.org/abs/2307.00073} {\path{arXiv:2307.00073}}.

\bibitem{SchemesCoq}
Laurent Chicli.
\newblock {Une formalisation des faisceaux et des sch{\'e}mas affines en
  th{\'e}orie des types avec Coq}.
\newblock Technical Report RR-4216, {INRIA}, June 2001.
\newblock URL: \url{https://hal.inria.fr/inria-00072403}.

\bibitem{ConstrSchemes}
Thierry Coquand, Henri Lombardi, and Peter Schuster.
\newblock Spectral schemes as ringed lattices.
\newblock {\em Annals of Mathematics and Artificial Intelligence},
  56(3):339--360, 2009.

\bibitem{EGA4.5}
Pierre Deligne and Jean-Fran\c{c}ois Boutot.
\newblock Cohomologie {\'e}tale: les points de d{\'e}part.
\newblock In {\em Cohomologie Etale}, pages 4--75, Berlin, Heidelberg, 1977.
  Springer Berlin Heidelberg.

\bibitem{DemazureGabriel}
Michel Demazure and Peter Gabriel.
\newblock {\em Introduction to algebraic geometry and algebraic groups}.
\newblock Elsevier, 1980.

\bibitem{EisenbudHarris}
David Eisenbud and Joe Harris.
\newblock {\em The geometry of schemes}, volume 197.
\newblock Springer Science \& Business Media, 2006.

\bibitem{Espanol83}
Luis Espa\~nol.
\newblock Le spectre d'un anneau dans l'alg\`ebre constructive et applications
  \`a la dimension.
\newblock {\em Cahiers de Topologie et G\'eom\'etrie Diff\'erentielle
  Cat\'egoriques}, 24(2):133--144, 1983.
\newblock URL: \url{http://www.numdam.org/item/CTGDC_1983__24_2_133_0/}.

\bibitem{GoertzWedhorn}
Ulrich G{\"o}rtz and Torsten Wedhorn.
\newblock {\em Algebraic geometry}.
\newblock Springer, 2010.

\bibitem{GrothendieckBuffalo}
Alexandre Grothendieck and Federico Gaeta.
\newblock {\em Introduction to Functorial Algebraic Geometry: After a Summer
  Course}.
\newblock State University of New York at Buffalo, Department of Mathematics,
  1974.
\newblock
  \url{https://ncatlab.org/nlab/files/GrothendieckIntrodFunctorialGeometryI1973.pdf}.

\bibitem{Hartshorne}
Robin Hartshorne.
\newblock {\em Algebraic geometry}, volume~52.
\newblock Springer Science \& Business Media, 2013.

\bibitem{HofmannStreicher1997}
Martin Hofmann and Thomas Streicher.
\newblock Lifting grothendieck universes.
\newblock
  \url{https://www2.mathematik.tu-darmstadt.de/~streicher/NOTES/lift.pdf},
  1997.
\newblock Unpublished note.

\bibitem{JohnstoneSketchesVolII}
Peter~T Johnstone.
\newblock {\em Sketches of an Elephant: A Topos Theory Compendium: Volume 2},
  volume~2.
\newblock Oxford University Press, 2002.

\bibitem{JoyalZarLat}
Andr\'{e} Joyal.
\newblock Les th\'{e}oremes de chevalley-tarski et remarques sur
  l’alg\`{e}bre constructive.
\newblock {\em Cahiers Topologie G\'{e}om. Diff\'{e}rentielle}, 16:256--258,
  1976.

\bibitem{KrausPropTrunc}
Nicolai Kraus.
\newblock The general universal property of the propositional truncation.
\newblock In Hugo Herbelin, Pierre Letouzey, and Matthieu Sozeau, editors, {\em
  20th International Conference on Types for Proofs and Programs (TYPES 2014)},
  volume~39 of {\em Leibniz International Proceedings in Informatics (LIPIcs)},
  pages 111--145, Dagstuhl, Germany, 2015. Schloss Dagstuhl--Leibniz-Zentrum
  fuer Informatik.
\newblock URL: \url{http://drops.dagstuhl.de/opus/volltexte/2015/5494}, \href
  {https://doi.org/http://dx.doi.org/10.4230/LIPIcs.TYPES.2014.111}
  {\path{doi:http://dx.doi.org/10.4230/LIPIcs.TYPES.2014.111}}.

\bibitem{LawvereMail}
F.~William Lawvere.
\newblock Categories: Grothendieck's 1973 buffalo colloquium, 2003.
\newblock Category mailing list, March 30.
\newblock URL: \url{https://www.mta.ca/~cat-dist/archive/2003/03-3}.

\bibitem{SheavesInGeometryAndLogic}
Saunders Mac~Lane and Ieke Moerdijk.
\newblock {\em Sheaves in geometry and logic: A first introduction to topos
  theory}.
\newblock Springer Science \& Business Media, 2012.

\bibitem{MadoreDefScheme}
David Madore.
\newblock Comment définir efficacement ce qu'est un schéma, 2013.
\newblock David Madore's WebLog.
\newblock URL:
  \url{http://www.madore.org/~david/weblog/d.2013-09-21.2160.definition-schema.html}.

\bibitem{MartinLof75itt}
Per Martin-L{\"o}f.
\newblock {An Intuitionistic Theory of Types: Predicative Part}.
\newblock In H.~E. Rose and J.~C. Shepherdson, editors, {\em Logic Colloquium
  '73}, volume~80 of {\em Studies in Logic and the Foundations of Mathematics},
  pages 73--118. North-Holland, 1975.
\newblock \href {https://doi.org/10.1016/S0049-237X(08)71945-1}
  {\path{doi:10.1016/S0049-237X(08)71945-1}}.

\bibitem{NPNotes}
Marc Nieper-Wi{\ss}kirchen.
\newblock Algebraische geometrie.
\newblock unpublished lecture notes, 2008.

\bibitem{StoneRepresentation}
Marshall~Harvey Stone.
\newblock Topological representations of distributive lattices and brouwerian
  logics.
\newblock {\em {\v{C}}asopis pro p{\v{e}}stov{\'a}n{\'\i} matematiky a fysiky},
  67(1):1--25, 1938.

\bibitem{Streicher93}
Thomas Streicher.
\newblock {\em Investigations Into Intensional Type Theory}.
\newblock Habilitation thesis, Ludwig-Maximilians-Universit\"at M\"unchen,
  1993.
\newblock URL:
  \url{https://www2.mathematik.tu-darmstadt.de/~streicher/HabilStreicher.pdf}.

\bibitem{HoTTBook}
The {Univalent Foundations Program}.
\newblock {\em Homotopy Type Theory: Univalent Foundations of Mathematics}.
\newblock \url{https://homotopytypetheory.org/book}, Institute for Advanced
  Study, 2013.

\bibitem{RisingSea}
Ravi Vakil.
\newblock The rising sea: Foundations of algebraic geometry.
\newblock {\em preprint}, 2017.
\newblock \url{https://math.stanford.edu/~vakil/216blog/FOAGnov1817public.pdf}.

\bibitem{CubicalAgda2}
Andrea Vezzosi, Anders M\"ortberg, and Andreas Abel.
\newblock Cubical agda: A dependently typed programming language with
  univalence and higher inductive types.
\newblock {\em Journal of Functional Programming}, 31:e8, 2021.
\newblock \href {https://doi.org/10.1017/S0956796821000034}
  {\path{doi:10.1017/S0956796821000034}}.

\bibitem{UniMath}
V.~Voevodsky, B.~Ahrens, D.~Grayson, et~al.
\newblock {\sc UniMath}: {Univalent} {Mathematics}.
\newblock Available at \url{https://github.com/UniMath}.

\bibitem{Voevodsky10bonn}
Vladimir Voevodsky.
\newblock Univalent foundations, September 2010.
\newblock Notes from a talk in Bonn.
\newblock URL:
  \url{https://www.math.ias.edu/vladimir/sites/math.ias.edu.vladimir/files/Bonn_talk.pdf}.

\bibitem{VoevodskyResizing}
Vladimir Voevodsky.
\newblock Resizing rules -- their use and semantic justification. slides from a
  talk at types, bergen, 11 september, 2011.

\bibitem{VoevodskyFoundationsLib}
Vladimir Voevodsky.
\newblock An experimental library of formalized mathematics based on the
  univalent foundations.
\newblock {\em Mathematical Structures in Computer Science}, 25(5):1278–1294,
  2015.
\newblock \href {https://doi.org/10.1017/S0960129514000577}
  {\path{doi:10.1017/S0960129514000577}}.

\bibitem{ZeunerMortberg23}
Max Zeuner and Anders M\"{o}rtberg.
\newblock {A Univalent Formalization of Constructive Affine Schemes}.
\newblock In Delia Kesner and Pierre-Marie P\'{e}drot, editors, {\em 28th
  International Conference on Types for Proofs and Programs (TYPES 2022)},
  volume 269 of {\em Leibniz International Proceedings in Informatics
  (LIPIcs)}, pages 14:1--14:24, Dagstuhl, Germany, 2023. Schloss Dagstuhl --
  Leibniz-Zentrum f{\"u}r Informatik.
\newblock URL: \url{https://drops.dagstuhl.de/opus/volltexte/2023/18457}, \href
  {https://doi.org/10.4230/LIPIcs.TYPES.2022.14}
  {\path{doi:10.4230/LIPIcs.TYPES.2022.14}}.

\end{thebibliography}
